\documentclass[11pt]{amsart}

\usepackage{import}
\usepackage[T1]{fontenc}
\usepackage[utf8]{inputenc}
\usepackage{lmodern}
\usepackage{amsmath,amssymb,mathtools}
\usepackage{tikz}
\usepackage[margin=1.15in]{geometry}
\usepackage[final]{microtype}
\usepackage[colorlinks=true,linkcolor=blue,citecolor=blue,urlcolor=blue]{hyperref}
\usepackage{enumitem}

\newtheorem{theorem}{Theorem}[section]
\newtheorem{lemma}[theorem]{Lemma}

\newtheorem{corollary}[theorem]{Corollary}
\newtheorem*{claim}{Claim}
\theoremstyle{definition}
\newtheorem{definition}[theorem]{Definition}
\theoremstyle{remark}
\newtheorem{remark}[theorem]{Remark}
\newtheorem{question}[theorem]{Question}

\numberwithin{equation}{section}

\newcommand{\rr}{\mathfrak{rr}}

\newcommand{\rrf}{\mathfrak{rr}_{f}}
\newcommand{\rri}{\mathfrak{rr}_{i}}
\newcommand{\rro}{\mathfrak{rr}_{o}}
\newcommand{\rrfi}{\mathfrak{rr}_{fi}}
\newcommand{\rrfo}{\mathfrak{rr}_{fo}}
\newcommand{\rrio}{\mathfrak{rr}_{io}}
\newcommand{\sr}{\mathfrak{sr}}
\newcommand{\sz}{\mathfrak{sz}}
\newcommand{\bb}{\mathfrak b}
\newcommand{\dd}{\mathfrak d}

\newcommand{\cM}{\mathcal M}
\newcommand{\cN}{\mathcal N}
\newcommand{\covN}{\operatorname{cov}(\cN)}
\newcommand{\cofM}{\operatorname{cof}(\cM)}

\newcommand{\nonM}{\operatorname{non}(\cM)}
\newcommand{\Sym}{\operatorname{Sym}(\omega)}
\newcommand{\CC}{\mathsf{CC}}

\newcommand{\abs}[1]{\left| #1\right|}

\title{Determining the Rearrangement Number}
\author{Vinicius de Oliveira Rodrigues}
\address{University of São Paulo, Rua do Matão, 1010, São Paulo, SP, Brazil}
\email{vinior@ime.usp.br}
\date{}

\hypersetup{
  pdftitle={Determining the Rearrangement Number},
  pdfauthor={Vinicius de Oliveira Rodrigues},
  pdfsubject={The rearrangement number and the uniformity of the meager ideal},
  pdfkeywords={rearrangement number, cardinal characteristics, slaloms, conditionally convergent series}
}


\begin{document}

\begin{abstract}
The rearrangement number $\rr$ is the least cardinality of a collection of
permutations of $\omega$ such that every conditionally convergent real
series is disrupted by some permutation in the collection.  Blass, Brendle,
Brian, Hamkins, Hardy, and Larson proved that $\max\{\covN,\bb\}\leq\rr\leq\nonM$
and asked whether $\rr=\nonM$.
We prove in ZFC that $\rr=\nonM$.
\end{abstract}

\maketitle

\section{Introduction}\label{sec:introduction}

Riemann's rearrangement theorem states that every conditionally
convergent real series has rearrangements converging to any prescribed
element of $\mathbb R\cup\{-\infty,+\infty\}$, as well as
rearrangements that diverge by oscillation \cite{Riemann}.  This
flexibility leads naturally to the following question: how many
permutations are required to disrupt every conditionally convergent
series?

To the best of the author's knowledge, this question first appeared
explicitly in a 2015 MathOverflow post by Hardy \cite{HardyMO}.  Blass,
Brendle, Brian, Hamkins, Hardy, and Larson subsequently introduced the
corresponding cardinal characteristic $\rr$, called the
\emph{rearrangement number}, and systematically studied it
\cite{BBBHHL}.  They proved
\[
 \max\{\covN,\bb\}\leq\rr\leq\nonM.
\]
The definitions of these cardinals are recalled in
Section~\ref{sec:preliminaries}.

The relevant characteristics occupy adjacent levels of Cichoń's
diagram, as illustrated in Figure~\ref{fig:intro-cichon}.  The two
lower bounds are incomparable in ZFC: both $\covN<\bb$ and
$\bb<\covN$ are consistent.

\begin{figure}[ht]
 \centering
 \begin{tikzpicture}[
   x=1cm,
   y=1cm,
   cardinal/.style={font=\small,inner sep=2pt},
   classical/.style={->,thin},
   rrbound/.style={->,thick,densely dashed},
   rrnode/.style={font=\large,draw,rounded corners,inner sep=4pt}
  ]
  \node[cardinal] (covN) at (0,2.8) {$\operatorname{cov}(\cN)$};
  \node[cardinal] (nonM) at (3,2.8) {$\operatorname{non}(\cM)$};
  \node[cardinal] (cofM) at (6,2.8) {$\operatorname{cof}(\cM)$};
  \node[cardinal] (cofN) at (9,2.8) {$\operatorname{cof}(\cN)$};
  \node[cardinal] (c) at (11.5,2.8) {$\mathfrak c$};

  \node[cardinal] (b) at (3,1.4) {$\bb$};
  \node[cardinal] (d) at (6,1.4) {$\dd$};

  \node[cardinal] (aleph1) at (-2.5,0) {$\aleph_1$};
  \node[cardinal] (addN) at (0,0) {$\operatorname{add}(\cN)$};
  \node[cardinal] (addM) at (3,0) {$\operatorname{add}(\cM)$};
  \node[cardinal] (covM) at (6,0) {$\operatorname{cov}(\cM)$};
  \node[cardinal] (nonN) at (9,0) {$\operatorname{non}(\cN)$};

  \draw[classical] (covN) -- (nonM);
  \draw[classical] (nonM) -- (cofM);
  \draw[classical] (cofM) -- (cofN);
  \draw[classical] (addN) -- (addM);
  \draw[classical] (addM) -- (covM);
  \draw[classical] (covM) -- (nonN);
  \draw[classical] (addN) -- (covN);
  \draw[classical] (addM) -- (b);
  \draw[classical] (b) -- (nonM);
  \draw[classical] (covM) -- (d);
  \draw[classical] (d) -- (cofM);
  \draw[classical] (nonN) -- (cofN);
  \draw[classical] (b) -- (d);
  \draw[classical] (aleph1) -- (addN);
  \draw[classical] (cofN) -- (c);

  \node[rrnode] (rr) at (1.55,1.95) {$\rr$};
  \draw[rrbound] (covN) -- (rr);
  \draw[rrbound] (b) -- (rr);
  \draw[rrbound] (rr) -- (nonM);
 \end{tikzpicture}
 \caption{Cichoń's diagram with the previously known bounds for $\rr$.}
 \label{fig:intro-cichon}
\end{figure}
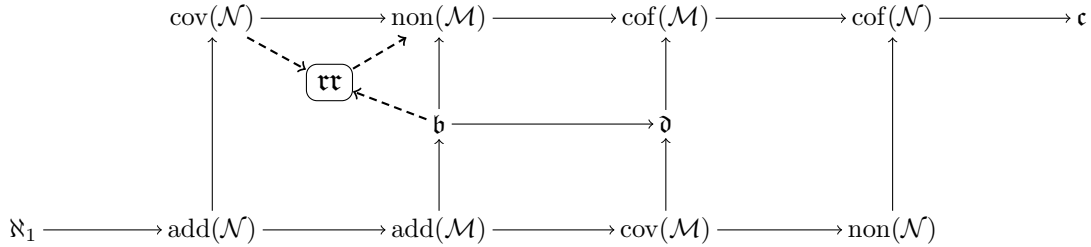

Explicitly, they asked the following.

\begin{question}[\cite{BBBHHL}, Question~46]\label{q:intro-rr-nonM}
Does ZFC prove that $\rr=\nonM$?
\end{question}

The question was repeated as Problem~10 by Brech, Brendle, and Telles
\cite{BrechBrendleTelles}.
In this paper, we settle this question by proving that the answer is positive.

\begin{theorem}[Main Theorem]\label{thm:main}
In ZFC, $\rr=\nonM$.
\end{theorem}

\section{Notation and previous results}\label{sec:preliminaries}

The set of all natural numbers is identified with the first infinite ordinal, $\omega$.
$\Sym$ denotes the group of all permutations of $\omega$.
We identify a real series with a sequence
$a=\langle a_n:n<\omega\rangle\in\mathbb R^\omega$ and write
\[
 S_m(a)=\sum_{n<m}a_n,
 \qquad
 S_m^\pi(a)=\sum_{n<m}a_{\pi(n)}
 \quad(\pi\in\Sym, m \in \omega).
\]
A series $a$ is conditionally convergent if $\sum_n a_n$ converges but
$\sum_n\abs{a_n}$ diverges.
The set of all conditionally convergent series is denoted by $\CC$.

\begin{definition}\label{def:rr}
A collection $\Pi\subseteq\Sym$ is \emph{rearranging} if for every
$a\in\CC$ there exists $\pi\in\Pi$ such that
$\sum_n a_{\pi(n)}$ either diverges or converges to a real number different
from the original sum.  The \emph{rearrangement number} is defined as
\[
      \rr=\min\{\abs{\Pi}:\Pi\subseteq\Sym
      \text{ is rearranging}\}.
\]
\end{definition}

For
$f,g\in\omega^\omega$, write $f\leq^*g$ when
$f(n)\leq g(n)$ for all but finitely many $n\in \omega$.
If $f\not \leq^*g$, we write $g<^\infty f$.

The bounding number
$\bb$ is the least cardinality of a $\leq^*$-unbounded subset of
$\omega^\omega$.

Let $\cM$ and $\cN$ be the ideals of meager and Lebesgue null subsets
of $\mathbb R$, respectively, and let $\mathfrak c=|\mathbb R|$.
The cardinal $\covN$ is the least cardinality of a collection of Lebesgue null subsets of $\mathbb R$ whose union is $\mathbb R$.
Finally, the cardinal $\nonM$ is the least cardinality of a nonmeager subset of $\mathbb R$.
It is well known that $\aleph_1\leq \max\{\covN,\bb\}\leq\nonM\leq \mathfrak c$.
See \cite{BlassHandbook} for more
background on these cardinal invariants.

\begin{theorem}[{\cite[Theorems 7, 8, and 11]{BBBHHL}}]
\label{thm:known-sandwich}
In ZFC, $\max\{\covN,\bb\}\leq\rr\leq\nonM$.
\end{theorem}
For a sequence $r:\omega\to\omega\setminus\{0\}$, let
\[
 {\mathcal S}_r=\left\{\varphi:\omega\to[\omega]^{<\omega}:
                         \abs{\varphi(n)}\leq r(n)\text{ for every }n\right\}.
\]
The elements of ${\mathcal S}_r$ are called \emph{$r$-slaloms}. A function
$x\in\omega^\omega$ \emph{eventually avoids} $\varphi$ if
$x(n)\notin\varphi(n)$ for all but finitely many $n$.

\begin{theorem}[Bartoszy\'nski, {\cite[Lemma 2.4.8]{BartoszynskiJudah}}]\label{thm:slaloms}
For every $r:\omega\to\omega\setminus\{0\}$, every family
$\Phi\subseteq{\mathcal S}_r$ of cardinality less than $\nonM$
has a common eventual avoider.
\end{theorem}
We identify a natural number $m$ with $\{0,\ldots,m-1\}$. For a finite set $X$, $\operatorname{Sym}(X)$
denotes its permutation group.
In particular, for a natural number $L$, $\operatorname{Sym}(L)$ is the standard permutation group $S_L$.

Following \cite{BlassHandbook}, we use the language of Galois--Tukey connections.

\begin{definition}
      A \emph{relation triple}, or simply a \emph{relation}, is a triple $\mathbf A=(A_-,A_+,A)$, where $A_-$ and $A_+$ are nonempty sets and $A\subseteq A_-\times A_+$ has domain $A_-$.

      Elements of $A_-$ are called \emph{challenges} and elements of $A_+$ are called \emph{responses}.

      The norm of a relation $\mathbf A$ is defined as
      \[
            \|\mathbf A\|=\min\{|Y|:Y\subseteq A_+\text{ and }\forall x\in A_-\ \exists y\in Y\ (x,y)\in A\}.
      \]
\end{definition}

Thus:
\begin{itemize}
      \item The cardinal $\mathfrak b$ is $\|\mathfrak B\|$, where $\mathfrak B=(\omega^\omega,\omega^\omega,<^{\infty})$.
      \item The cardinal $\nonM$ is $\|\mathfrak S_r\|$, where $\mathfrak S_r=(\omega^\omega,{\mathcal S}_r,\in^\infty)$ for any $r:\omega\to\omega\setminus\{0\}$ with $\lim_{n\to\infty}r(n)=\infty$.
      Here $e\in^\infty\varphi$ means that $e(n)\in\varphi(n)$ for infinitely many $n$.
      \item The cardinal $\rr$ is $\|\mathfrak R\|$, where $\mathfrak R=(\CC,\Sym,\text{rearranges})$.
\end{itemize}

\begin{definition}
      Let $\mathbf A=(A_-,A_+,A)$ and $\mathbf B=(B_-,B_+,B)$ be relations. A \emph{morphism} from $\mathbf A$ to $\mathbf B$ is a pair of functions $(\phi_-,\phi_+)$, where $\phi_-:B_-\to A_-$ and $\phi_+:A_+\to B_+$, such that
      \[
            \forall x\in B_-\ \forall y\in A_+\ (\phi_-(x),y)\in A\rightarrow(x,\phi_+(y))\in B.
      \]
\end{definition}

It is straightforward to check that if there is a morphism from $\mathbf A$ to $\mathbf B$, then $\|\mathbf A\|\geq\|\mathbf B\|$.

\begin{definition}[{\cite[Definition 4.10]{BlassHandbook}}]
\label{def:sequential-composition}
For relations $\mathbf A=(A_-,A_+,A)$ and $\mathbf B=(B_-,B_+,B)$,
their \emph{sequential composition} is
\[
 \mathbf A;\mathbf B
 =\bigl(A_-\times{}^{A_+}B_-,\ A_+\times B_+,\ C\bigr),
\]
where ${}^{A_+}B_-$ is the set of functions from $A_+$ to $B_-$, and
\[
 (x,f)\,C\,(a,b)
 \quad\Longleftrightarrow\quad x\,A\,a\text{ and }f(a)\,B\,b.
\]
Thus the second challenge $f(a)$ depends on the response $a$ to the
first challenge $x$.
\end{definition}

\begin{remark}[{\cite[Theorem 4.11(3)]{BlassHandbook}}]
\label{rem:sequential-norm}
The norm of the sequential composition satisfies
\[
 \|\mathbf A;\mathbf B\|=\|\mathbf A\|\cdot\|\mathbf B\|.
\]
For the upper bound, which is all we need, take families $Y\subseteq A_+$
and $Z\subseteq B_+$ witnessing the two norms. Given a challenge $(x,f)$,
choose $a\in Y$ with $x\,A\,a$, and then $b\in Z$ with $f(a)\,B\,b$.
Hence $Y\times Z$ responds to every challenge in $\mathbf A;\mathbf B$.
\end{remark}

\section{A counting estimate}\label{sec:finite}

For a finitely supported sequence $u\in\mathbb R^\omega$ and
$\pi\in\Sym$, let
\[
 M(u)=\sup_{j<\omega}\left|S_j(u)\right| \quad \text{and}\quad
 M_\pi(u)=\sup_{j<\omega}\left|S_j^\pi(u)\right|.
\]
For $v\in\mathbb R^L$ and $\sigma\in\operatorname{Sym}(L)$, use
the same notation with the maximum over $j\leq L$.

When $v\in \{-p, p\}^L$ for some $p>0$, we think of $v$ as a walk starting at $0$.
We interpret $v_i=p$ as a step of length $p$ to the right and $v_i=-p$ as a step of length $p$ to the left.
Then $M(v)$ is the maximum distance from the origin reached by this walk.

We say that $v \in \{-p, p\}^L$ is \emph{balanced} if $S_L(v)=0$, that is, if the walk ends at the origin.
Clearly, balanced vectors exist only for even integers $L$.

\begin{lemma}\label{lem:finite}
For all positive integers $m$ and $q$,
there exist a positive integer $L$ and balanced vectors
$(v_k)_{k<m}$ in $\{-1/L,1/L\}^L$ satisfying, for every $\sigma\in\operatorname{Sym}(L)$,
\begin{equation*}
 \left|\{k<m:M_\sigma(v_k)>1/q\}\right|\leq 4q^4.
\end{equation*}
\end{lemma}

Recall that if $(e_k)_{k<m}$ is an orthonormal family in a real
inner product space, then Bessel's inequality gives
$\sum_{k<m}|\langle b,e_k\rangle|^2\leq\|b\|^2$ for every vector $b$
in that space.

\begin{proof}
List the elements of $\{-1,1\}^m$ as $(x_i)_{i<L}$, where $L=2^m$.
For $k<m$, put
\[
 v_k(i)=\frac{x_i(k)}L\qquad(i<L).
\]
Each $v_k$ is balanced, since
\begin{equation*}
      \sum_{i<L}v_k(i)=\frac1L\sum_{i<L}x_i(k)
      =\frac1L\sum_{x\in\{-1,1\}^m}x(k)=0.
\end{equation*}
Moreover, for every two distinct $k,\ell<m$,
\begin{equation}\label{eq:orthogonal}
 v_k\cdot v_\ell
 =\frac1{L^2}\sum_{x\in\{-1,1\}^m}x(k)x(\ell)
 =\frac1{L^2}\left(\sum_{x(k)=x(\ell)}1-\sum_{x(k)\neq x(\ell)}1\right)
 =0.
\end{equation}
Also, for every $k<m$,
\begin{equation}\label{eq:norm}
 \|v_k\|^2=\sum_{i<L}v_k(i)^2
 =\sum_{i<L}\frac1{L^2}=\frac1L.
\end{equation}

Fix $\sigma\in\operatorname{Sym}(L)$.
The vectors $e_k=\sqrt L\,(v_k\circ\sigma)$ are orthonormal by \eqref{eq:orthogonal} and \eqref{eq:norm}.
For $j\leq L$, define $b_j\in\mathbb R^L$ by
\[
 b_j(i)=
 \begin{cases}
  1,&i<j,\\
  0,&j\leq i<L.
 \end{cases}
\]
Then, for each $j\leq L$,
\begin{equation}\label{eq:prefix-square}
 \sum_{k<m}|S_j^\sigma(v_k)|^2
 =\frac1L\sum_{k<m}|\langle b_j,e_k\rangle|^2
 \leq\frac{\|b_j\|^2}{L}
 =\frac jL\leq1.
\end{equation}

For $k<m$ and $j\leq L$, telescoping gives
\[
\begin{aligned}
 |S_j^\sigma(v_k)|^2
 &=\sum_{i<j}\left((S_{i+1}^\sigma(v_k))^2
                         -(S_i^\sigma(v_k))^2\right)\\
 &=\sum_{i<j}\left(S_{i+1}^\sigma(v_k)-S_i^\sigma(v_k)\right)
             \left(S_{i+1}^\sigma(v_k)+S_i^\sigma(v_k)\right)\\
 &=\sum_{i<j}v_k(\sigma(i))
             \left(S_{i+1}^\sigma(v_k)+S_i^\sigma(v_k)\right)\\
 &\leq\frac1L\sum_{i<L}
             \left(|S_{i+1}^\sigma(v_k)|+|S_i^\sigma(v_k)|\right)
 =\frac2L\sum_{i<L}|S_i^\sigma(v_k)|,
\end{aligned}
\]
where the last equality uses $S_0^\sigma(v_k)=S_L^\sigma(v_k)=0$.
Taking the maximum over $j$ and applying the Cauchy--Schwarz inequality, we get
\[
 M_\sigma(v_k)^4
 \leq\left(\frac2L\sum_{i<L}1\cdot|S_i^\sigma(v_k)|\right)^2
 \leq\frac4L\sum_{i<L}|S_i^\sigma(v_k)|^2.
\]
Summing over $k$ and using \eqref{eq:prefix-square},
\[
 \sum_{k<m}M_\sigma(v_k)^4
 \leq\frac4L\sum_{i<L}\sum_{k<m}|S_i^\sigma(v_k)|^2
 \leq\frac4L\sum_{i<L}1=4.
\]
Finally, let $B=\{k<m:M_\sigma(v_k)>1/q\}$.
The required bound follows from
\[
 \frac{|B|}{q^4}=\sum_{k\in B}\frac1{q^4}
 \leq\sum_{k\in B}M_\sigma(v_k)^4
 \leq\sum_{k<m}M_\sigma(v_k)^4\leq4.\qedhere
\]
\end{proof}

\section{A morphism into the slaloms}\label{sec:construction}

We construct a morphism from $\mathfrak B;\mathfrak R$ to
$\mathfrak S_r$ for a suitable sequence $r$. This will give
$\nonM\leq\max\{\mathfrak b,\rr\}$.

\begin{lemma}\label{lem:morphism}
There is a function $r:\omega\to\omega\setminus\{0\}$ with
$r(n)\to\infty$ and a morphism
\[
 (\phi_-,\phi_+):\mathfrak B;\mathfrak R\longrightarrow\mathfrak S_r.
\]
\end{lemma}

\begin{proof}
Put $r(n)=8(2^{n+1})^4$ for $n<\omega$.
We shall define:
\[
\begin{aligned}
 \phi_-&:\omega^\omega\longrightarrow
              \omega^\omega\times{}^{\omega^\omega}\CC,\\
 \phi_+&:\omega^\omega\times\Sym\longrightarrow\mathcal S_r.
\end{aligned}
\]

For each $g\in\omega^\omega$ and $n\in \omega$ with $g(n)>0$, apply Lemma~\ref{lem:finite}
with $m=g(n)$ and $q=2^{n+1}$ to obtain a positive integer $L_n^g$ and
balanced vectors $(v_{n,k}^g)_{k<g(n)}$ in $\{-1/L_n^g,1/L_n^g\}^{L_n^g}$
such that, for every $\sigma\in\operatorname{Sym}(L_n^g)$,
\begin{equation}\label{eq:finite-block-bound}
 \left|\{k<g(n):M_\sigma(v_{n,k}^g)>1/2^{n+1}\}\right|
 \leq 4(2^{n+1})^4=\frac{r(n)}2.
\end{equation}
If $g(n)=0$, put $L_n^g=1$.

Partition $\omega$ into successive intervals $(I_n^g:n<\omega)$
of lengths $L_n^g$. For $k<g(n)$, define $u_{n,k}^g$ by placing $v_{n,k}^g$ on $I_n^g$:
\[
 u_{n,k}^g(i)=
 \begin{cases}
  v_{n,k}^g(i-\min I_n^g),&i\in I_n^g,\\
  0,&i\notin I_n^g.
 \end{cases}
\]
Thus $u_{n,k}^g$ is supported on $I_n^g$, with sum zero and absolute
sum one.

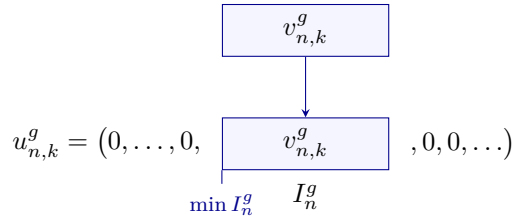
\begin{figure}[htbp]
\centering
\begin{tikzpicture}[font=\small, >=stealth,
 block/.style={draw=blue!55!black,fill=blue!4,
 minimum width=2.2cm,minimum height=0.65cm}]
 \node[block] (v) at (0,1.5) {$v_{n,k}^g$};
 \node[block] (u) at (0,0) {$v_{n,k}^g$};
 \draw[->,blue!55!black] (v.south) -- (u.north);
 \node[anchor=east] at (-1.25,0) {$u_{n,k}^g=\bigl(0,\ldots,0,$};
 \node[anchor=west] at (1.25,0) {$,0,0,\ldots\bigr)$};
 \node[below] at (0,-0.4) {$I_n^g$};
 \draw[blue!55!black] (u.south west) -- ++(0,-0.18)
  node[below,font=\scriptsize] {$\min I_n^g$};
\end{tikzpicture}
\caption{Placing $v_{n,k}^g$ on $I_n^g$ and extending by zero.}
\label{fig:block-embedding}
\end{figure}

For $\pi\in\Sym$, define the response map by
\begin{equation}\label{eq:bad-slalom}
 \phi_+(g,\pi)(n)=
 \{k<g(n):M(u_{n,k}^g)>1/2^{n+1}
                  \text{ or }M_\pi(u_{n,k}^g)>1/2^{n+1}\}.
\end{equation}

\begin{claim}
For every $g\in\omega^\omega$ and $\pi\in\Sym$, the sequence
$\phi_+(g,\pi)$ is an $r$-slalom.
\end{claim}

\begin{proof}[Proof of the claim]
Fix $n<\omega$. If $g(n)=0$, then $\phi_+(g,\pi)(n)=\varnothing$,
so the required bound is immediate. Assume $g(n)>0$.
Let $\pi_n^{g}\in\operatorname{Sym}(L_n^g)$ list the coordinates of
$I_n^g$ in the order in which $\pi$ visits them. Write
\[
    \pi^{-1}[I_n^g]=\{t_0<\cdots<t_{L_n^g-1}\}
\]
and define
\[
    \pi_n^g(i)=\pi(t_i)-\min I_n^g\qquad(i<L_n^g).
\]
Thus, the mapping
$s\mapsto\pi^{-1}(\min I_n^g+\pi_n^{g}(s))$ from $L_{n}^g$ to $\pi^{-1}[I_n^g]$ is a strictly increasing bijection.
For this reason, for every $k<g(n)$ and $j<\omega$,
\[
 S_j^\pi(u_{n,k}^g)=S_{|\pi[j]\cap I_n^g|}^{\pi_n^{g}}(v_{n,k}^g),
\]
and therefore $M_\pi(u_{n,k}^g)=M_{\pi_n^{g}}(v_{n,k}^g)$.
Thus, by \eqref{eq:finite-block-bound},
\[
    |\{k<g(n):M_{\pi_n^{g}}(v_{n,k}^g)>1/2^{n+1}\}|
    \leq 4(2^{n+1})^4=\frac{r(n)}2.
\]
Similarly, $M(u_{n,k}^g)=M(v_{n,k}^g)$, so the same
estimate applies to the first condition in \eqref{eq:bad-slalom}. Hence
\begin{equation*}
 |\phi_+(g,\pi)(n)|\leq\frac{r(n)}2+\frac{r(n)}2=r(n).
\end{equation*}
\end{proof}

Next, for $e\in\omega^\omega$, let $A_e^g=\{n<\omega:e(n)<g(n)\}$ and define, coordinatewise,
\begin{equation*}
 a_e^g=\sum_{n\in A_e^g}u_{n,e(n)}^g.
\end{equation*}
Notice that each $i<\omega$ belongs to
a unique interval $I_n^g$, and every summand supported on another
interval vanishes at $i$. Thus, for this $n$, which depends on $i$, we have
\[
 a_e^g(i)=
 \begin{cases}
  u_{n,e(n)}^g(i),&n\in A_e^g,\\
  0,&n\notin A_e^g,
 \end{cases}
 \qquad\text{where }i\in I_n^g.
\]
In particular, $a_e^g$ is well-defined.

\begin{claim}
If $e<^\infty g$ and $e\notin^\infty\phi_+(g,\pi)$, then
$a_e^g\in\CC$ and both its natural sum and its $\pi$-rearranged sum
are zero.
\end{claim}

\begin{proof}[Proof of the claim]
Since $e<^\infty g$, the set $A_e^g$ is infinite, so
\begin{equation}\label{eq:absolute}
 \sum_{i\in \omega}|a_e^g(i)|
 =\sum_{n\in A_e^g}\sum_{i\in I_n^g}|u_{n,e(n)}^g(i)|
 =\sum_{n\in A_e^g}1=\infty.
\end{equation}
Fix $\rho\in\{\operatorname{id}_\omega,\pi\}$ and $\eta>0$.
Choose $N$ such that $2^{-N}<\eta$ and
$e(n)\notin\phi_+(g,\pi)(n)$ for every $n\geq N$.
For $n\in A_e^g$ with $n\geq N$, \eqref{eq:bad-slalom} gives
$M_\rho(u_{n,e(n)}^g)\leq1/2^{n+1}$.
For each $j$, the finite set $\rho[j]$ meets only finitely many of
the disjoint intervals $I_n^g$. Grouping the terms of the partial sum
according to these intervals gives
\[
\begin{aligned}
 S_j^\rho(a_e^g)
 &=\sum_{i\in\rho[j]}a_e^g(i)\\
 &=\sum_{n\in A_e^g}\sum_{i\in\rho[j]\cap I_n^g}u_{n,e(n)}^g(i)
 =\sum_{n\in A_e^g}S_j^\rho(u_{n,e(n)}^g).
\end{aligned}
\]
Only finitely many summands in the outer sums are nonzero.

Since $\bigcup_{n<N}I_n^g$ is finite, choose $J$ such that
$\bigcup_{n<N}I_n^g\subseteq\rho[J]$.
For every $j\geq J$ and $n\in A_e^g$ with $n<N$,
$S_j^\rho(u_{n,e(n)}^g)=\sum_{i\in I_n^g}u_{n,e(n)}^g(i)=0$.

Thus, for $j\geq J$ we have
\[
\begin{aligned}
 |S_j^\rho(a_e^g)|
 &=\left|\sum_{\substack{n\in A_e^g\\n\geq N}}
                  S_j^\rho(u_{n,e(n)}^g)\right|\\
 &\leq\sum_{\substack{n\in A_e^g\\n\geq N}}
                   |S_j^\rho(u_{n,e(n)}^g)|\\
 &\leq\sum_{\substack{n\in A_e^g\\n\geq N}}M_\rho(u_{n,e(n)}^g)\\
 &\leq\sum_{n\geq N}\frac1{2^{n+1}}=2^{-N}<\eta.
\end{aligned}
\]
Since $\eta>0$ was arbitrary, $S_j^\rho(a_e^g)\to0$.
Taking $\rho=\operatorname{id}_\omega$ and $\rho=\pi$ gives the two
asserted sums, and \eqref{eq:absolute} gives conditional convergence.
\end{proof}

To define the challenge map, fix any $a^*\in\CC$ and put
\[
 \phi_-(e)=(e,f_e),\qquad
 f_e(g)=
 \begin{cases}
  a_e^g,&a_e^g\in\CC,\\
  a^*,&\text{otherwise}.
 \end{cases}
\]
Finally, suppose $(g,\pi)$ responds to $\phi_-(e)$. This means
$e<^\infty g$ and $\pi$ rearranges $f_e(g)$.
If $e\notin^\infty\phi_+(g,\pi)$, the preceding claim gives
$a_e^g\in\CC$ and shows that $\pi$ preserves its sum.
But then $f_e(g)=a_e^g$, a contradiction.
Therefore $e\in^\infty\phi_+(g,\pi)$, verifying the morphism condition.
\end{proof}

Consequently, Theorem~\ref{thm:main} holds.
\begin{proof}[Proof of Theorem~\ref{thm:main}]
By Lemma~\ref{lem:morphism}, there exists $r:\omega\to\omega\setminus\{0\}$ with $r(n)\to\infty$ and a morphism
\[
 (\phi_-,\phi_+):\mathfrak B;\mathfrak R\longrightarrow\mathfrak S_r.
\]
By Remark~\ref{rem:sequential-norm} and the known bound
$\mathfrak b\leq\rr$ from Theorem~\ref{thm:known-sandwich},
\[
\begin{aligned}
 \nonM=\|\mathfrak S_r\|
 &\leq\|\mathfrak B;\mathfrak R\|\\
 &=\max\{\|\mathfrak B\|,\|\mathfrak R\|\}
 =\max\{\mathfrak b,\rr\}=\rr.
\end{aligned}
\]
The reverse inequality $\rr\leq\nonM$ is also part of
Theorem~\ref{thm:known-sandwich}.
\end{proof}

\section{Conclusion}\label{sec:conclusion}

Theorem~\ref{thm:main} answers Question~46 of \cite{BBBHHL}, repeated
as Problem~10 of \cite{BrechBrendleTelles}, by showing that $\mathfrak{rr}=\nonM$.
In this section, we record some further
consequences for other related cardinal characteristics of the continuum.

In \cite{BBBHHL}, the authors defined several variants of the rearrangement number
by restricting the permitted outcomes of rearrangements.

\begin{definition}\label{def:rr-variants}
If $\pi\in\Sym$
rearranges $a\in\CC$, we say that the outcome of the rearrangement is of type
\begin{enumerate}
\item[$f$] if $\sum_n a_{\pi(n)}$ converges to a real number different
      from the original sum;
\item[$i$] if it diverges to $+\infty$ or to $-\infty$;
\item[$o$] if its partial sums diverge by oscillation, meaning that
      they have no limit in the extended real line
      $[-\infty,+\infty]$.
\end{enumerate}
If $x$ is one of $f,i,o,fi,fo,io$, a collection
$\Pi\subseteq\Sym$ is \emph{$x$-rearranging} if, for every
$a\in\CC$, some $\pi\in\Pi$ has one of the outcome types whose letter
occurs in $x$. Define $\mathfrak{rr}_x$ as the least cardinality of an
$x$-rearranging family. This gives the cardinals
\[
 \rrf,\quad\rri,\quad\rro,\quad
 \rrfi,\quad\rrfo,\quad\rrio.
\]
\end{definition}

Some ZFC relations between these cardinals were obtained in \cite{BBBHHL} and are shown in Figure~\ref{fig:rr-variants}.

\begin{figure}[ht]
\begin{center}
\begin{tikzpicture}[
  x=1cm,
  y=1cm,
  cardinal/.style={font=\large,inner sep=2pt}
]
  \node[cardinal] (rrf)  at (-1.5,1.5) {$\rrf$};
  \node[cardinal] (rri)  at ( 1.5,1.5) {$\rri$};
  \node[cardinal] (rrfi) at ( 0,0)     {$\rrfi$};
  \node[cardinal] (rr)   at ( 0,-1.5)  {$\rr=\rrfo=\rrio=\rro$};

  \draw (rrf) -- (rrfi);
  \draw (rri) -- (rrfi);
  \draw (rrfi) -- (rr);
\end{tikzpicture}
\end{center}
\caption{The hierarchy of rearrangement numbers.
Higher means provably at least as large in ZFC.}
\label{fig:rr-variants}
\end{figure}
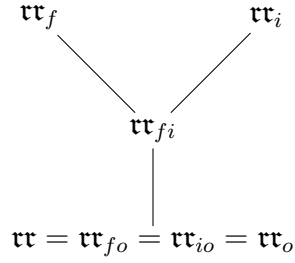

Recall that $\dd$ is the least size of a dominating family in
$(\omega^\omega,\leq^*)$, and $\cofM$ is the least size of a family
of meager sets cofinal under inclusion in the ideal $\cM$.

\begin{corollary}\label{cor:rr-variants}
In ZFC,
\[
 \rr=\rro=\rrfo=\rrio=\nonM
 \quad\text{and}\quad
 \cofM\leq\rrfi\leq\min\{\rrf,\rri\}.
\]
In particular, if $\cofM=\mathfrak c$, then
$\rrfi=\rrf=\rri=\mathfrak c$.
\end{corollary}

\begin{proof}
By \cite[Theorem~5]{BBBHHL}, $\rr=\rro=\rrfo=\rrio$, so the first
equality follows from Theorem~\ref{thm:main}. Also,
$\dd\leq\rrfi$ by \cite[Theorem~12]{BBBHHL}. Thus the classical
identity $\cofM=\max\{\dd,\nonM\}$
\cite[Theorem~5.6]{BlassHandbook} gives
\[
 \cofM=\max\{\dd,\rr\}\leq\rrfi.
\]
The remaining conclusions follow from the definitions and the upper
bound $\mathfrak c$ for all the rearrangement numbers.
\end{proof}

Some cardinals related to the rearrangement number were defined
in \cite{BrendleBrianHamkins}, including the subseries number $\sz$,
the subrearrangement number $\sr$, and the oscillation subseries number $\sz_o$.

\begin{definition}[{\cite{BrendleBrianHamkins}}]\label{def:subseries}
The \emph{subseries number} $\sz$ is the least size of a family
$\mathcal A\subseteq[\omega]^\omega$ such that, for every $a\in\CC$ there exists $A\in\mathcal A$ such that $\sum_{n\in A}a_n$ diverges.
Restricting divergence to divergence by oscillation gives the \emph{oscillation subseries number} $\sz_o$.

The \emph{subrearrangement number} $\sr$ is the least size of a
family $F$ of injections from $\omega$ to $\omega$ such that, for
every $a\in\CC$, some $f\in F$ makes $\sum_n a_{f(n)}$ diverge.
\end{definition}

Brendle, Brian, and Hamkins proved
$\covN\leq\sz\leq\sz_o\leq\nonM$
\cite[Theorems~4, 11, and~13]{BrendleBrianHamkins}.
They also established $\rr\leq\max\{\bb,\sz\}$ and
$\sz_o\leq\max\{\bb,\sz\}$
\cite[Theorems~19(1) and~21]{BrendleBrianHamkins}, while
$\bb\leq\rr$ was already known from \cite[Theorem~11]{BBBHHL}.
For subrearrangements, they obtained the identity
$\sr=\min\{\sz,\rr\}$
\cite[Theorem~35]{BrendleBrianHamkins}.

The comparison between $\sz$ and $\rr$ was left undecided in one
direction: Question~20 of \cite{BrendleBrianHamkins} asks whether
$\rr<\sz$ is consistent. Equivalently, in view of $\bb\leq\rr$,
they asked whether the bound $\rr\leq\max\{\bb,\sz\}$ can be strict.
At the end of Section~10, they also asked whether $\sr<\sz$ is
consistent. In the other direction, their analysis of the Laver model
already gave the consistency of $\sz=\sz_o<\bb=\rr$
\cite[Section~9]{BrendleBrianHamkins}.
Combining the known relations with Theorem~\ref{thm:main} yields the
following answers.

\begin{corollary}\label{cor:subseries}
In ZFC,
\[
 \sr=\sz\leq\sz_o\leq\rr=\nonM,
 \qquad
 \rr=\max\{\bb,\sz\}=\max\{\bb,\sz_o\}.
\]
If $\bb<\nonM$, then $\sr=\sz=\sz_o=\rr=\nonM$.
\end{corollary}

\begin{proof}
By \cite[Theorems~13 and~19(1)]{BrendleBrianHamkins},
$\sz_o\leq\nonM$ and $\rr\leq\max\{\bb,\sz\}$.
Together with $\sz\leq\sz_o$, $\bb\leq\rr$, and
Theorem~\ref{thm:main}, these yield
\[
 \rr\leq\max\{\bb,\sz\}\leq\max\{\bb,\sz_o\}\leq\rr.
\]
Moreover, \cite[Theorem~35]{BrendleBrianHamkins} gives
$\sr=\min\{\sz,\rr\}=\sz$.
If $\bb<\nonM=\rr$, the identity $\rr=\max\{\bb,\sz\}$ forces
$\sz=\rr$, and the other equalities follow.
\end{proof}

In particular, Question~20 of \cite{BrendleBrianHamkins}, asking
whether $\rr<\sz$ is consistent, has a negative answer. The same is
true of the question at the end of Section~10 of that paper about
the consistency of $\sr<\sz$.

\subsection*{Acknowledgments}

The author thanks Tan Özalp for pointing out a serious error concerning the use of
the Laver property in a forcing notion in an earlier version of this manuscript.

This study was financed, in part, by the São Paulo Research Foundation
(FAPESP), Brazil. Process Number 2025/07302-0.

\subsection*{Declaration on the use of generative artificial intelligence}

The author used the OpenAI GPT-6 Astra large language model through Codex during the
research and preparation of this manuscript for generating and
developing novel mathematical arguments and proofs, searching the literature,
checking arguments for possible gaps and errors, and drafting and
revising the text and \LaTeX{} source. The author
takes full responsibility for the arguments, claims, references, and
any errors that remain.

\bibliographystyle{amsplain}
\bibliography{includes/references}

\end{document}